\documentclass[12pt, reqno]{amsart}
\usepackage{amsmath, amsthm, amscd, amsfonts, amssymb, graphicx, xcolor}
\usepackage[hidelinks]{hyperref}

\newtheorem{theorem}{Theorem}[section]
\newtheorem{lemma}[theorem]{Lemma}

\newtheorem{corollary}[theorem]{Corollary}
\theoremstyle{definition}
\newtheorem{definition}[theorem]{Definition}
\newtheorem{example}[theorem]{Example}

\theoremstyle{remark}
\newtheorem{remark}[theorem]{Remark}
\numberwithin{equation}{section}

\begin{document}
\setcounter{page}{1}

\color{darkgray}
\noindent 
{\small The University of Notre Dame} \hfill {\small Sept. 2026}

\centerline{}

\centerline{}

\title{Finite Spectra of Syllogistic Logic with Cardinality Comparisons}

\author{Ruiting Jiang}

\address{
Department of Philosophy\\
University of Notre Dame\\
Notre Dame, IN, USA
}

\email{rjiang2@nd.edu}

\begin{abstract}
We study the finite spectra of the syllogistic logic with cardinality
comparisons $S^\dagger(card)$. Since the language does not contain
conjunction at the sentence level, we consider the spectrum of a theory
$\Gamma$, namely the set of positive integers $n$ for which $\Gamma$ is
satisfiable in an $n$-element model. We give a complete classification
of these spectra: besides the empty set, every $S^\dagger(card)$-spectrum
is either an eventual tail of the positive integers or an eventual tail
of the positive even integers. We then consider the extension of
$S^\dagger(card)$ by Boolean connectives at the sentence level and show
that the collection of its spectra forms the topology generated by the
spectra of the original language.
\end{abstract}

\maketitle 


\section{Introduction}

\noindent The finite spectrum of a sentence $\phi$ is defined\footnote{See, for example, Marker~\cite[p.~30]{Marker2002}.} as the set of positive integers $n$ for which $\phi$ is satisfiable in an $n$-element model. 

\begin{definition}
	$Spec(\phi) = \{n\in\mathbb N^+ : \phi\text{ has a model of cardinality }n\}$. 
\end{definition}

In 1952, Scholz~\cite{Scholz1952} asked which subsets of $\mathbb N^+$ are the finite spectra of first-order sentences. An answer to this question can take many forms. For example, Jones and Selman~\cite{JonesSelman1974} give a complexity-theoretic characterization of first-order spectra, showing that they are exactly sets accepted by nondeterministic Turing machines in exponential time.

In fact, this question can be generalised to any formal language $\mathcal{L}$, and the answer provides insights about which finite cardinalities can be distinguished by the language $\mathcal{L}$. Therefore, this is an important question about the expressive power of a language, and in the case of first-order language it is an active topic of research in finite model theory. For more literature on finite spectra, see Durand et al.~\cite{SpectrumSurvey} for a survey.

In this paper, we ask this question for the Syllogistic Logic with
Cardinality Comparisons $S^\dagger(card)$, introduced by Moss~\cite{Larry}.
This is a particularly natural setting for the spectrum question. The
language is syntactically weak: its sentences can only say things such
as ``all $x$'s are $y$'s'', ``some $x$'s are $y$'s'', ``there are at
least as many $x$'s as $y$'s'', and ``there are more $x$'s than
$y$'s'', where nouns are interpreted as subsets of the underlying
finite universe. At the same time, the latter two forms of sentence
compare cardinalities directly. Hence it is natural to ask how much
information about the cardinality of the entire universe can be
recovered from comparisons between the cardinalities of its subsets.

However, since there is no conjunction at the sentence level, we do not ask about the spectrum of a single sentence. Instead, we consider the spectrum of a theory\footnote{By a ``theory'', we just mean a set of sentences, not necessarily closed under logical consequence.} $\Gamma$: 

\begin{definition}
	$Spec(\Gamma) = \{n\in\mathbb N^+ : \Gamma\text{ has a model of cardinality }n\}$.
\end{definition}

We call such a set an $S^{\dagger}(card)$-spectrum. When $\Gamma$ is finite, this is the natural analogue of the ordinary finite spectrum in a language with conjunction, because $\Gamma$ is equivalent to the conjunction of its sentences. Moreover, as we will see, for any theory $\Gamma$ in $S^{\dagger}(card)$, there is a finite theory $\Gamma_0$ such that $Spec(\Gamma) = Spec(\Gamma_0)$.  

Our main result gives a complete classification of the $S^{\dagger}(card)$-spectra. For $(k\geq1)$, let $T_k=\{n\in\mathbb{N}^+ \mid n\geq k\}$ be the eventual tail of the natural numbers, and $E_k=\{n\in\mathbb{N}^+ \mid n\geq k
\text{ and }n\text{ is even}\}$ be the eventual tail of the even natural numbers.

\begin{theorem}
	The $S^{\dagger}(card)$-spectra are exactly the sets $\varnothing, T_k,$ and $E_k$, where $k\geq1$.
\end{theorem}

Thus $S^{\dagger}(card)$ can express two kinds of information about the size of a finite universe: a lower bound, and even cardinality. 

The proof proceeds by constructing, for each positive integer $n$, a
canonical theory $\Gamma_n$ which records every fact expressible
in $S^\dagger(card)$ about a canonical $n$-element interpretation in
which every subset of the universe is named by a noun. We then determine
which cardinalities belong to $Spec(\Gamma_n)$. We show
that when $n$ is odd, every $m\geq n$ belongs to
$Spec(\Gamma_n)$, while when $n$ is even, every even
$m\geq n$ does. A final argument shows that if a spectrum contains an
even number $m$ and a larger odd number, then it also contains all
numbers between them. These facts yield the classification
theorem. 

We also study an extension of $S^{\dagger}(card)$, obtained by adding Boolean connectives at the sentence level of the original language. We show that the collection of spectra in this extension is exactly the topology on $\mathbb N^+$ generated by the spectra of the original language $S^{\dagger}(card)$.

The present question is related to, but should be distinguished from, frame correspondence\footnote{See, for example, Blackburn, de Rijke, and Venema~\cite{ML}.}. They are related in the sense that both study the relation between the satisfaction of a formula and the structure of the model. However, they are different in their approach. In modal correspondence theory, a formula typically characterizes a frame by being valid under every valuation on that frame. Here the relevant quantification over interpretations is existential: $n$ belongs to the spectrum of $\Gamma$ when there exists an interpretation on an $n$-element universe satisfying $\Gamma$. 

The remainder of the paper is organized as follows. Section 2 recalls the syntax and semantics of $S^\dagger(card)$. Section 3 develops the characteristic theories and proves the spectrum classification theorem. Section 4 studies the addition of Boolean connectives at the sentence level. Section 5 discusses the source of the parity phenomenon, possible extensions of the language, and the connection between the characteristic theories used here and diagram constructions in model theory.


\newpage

\section{Preliminaries}

We first give the syntax and semantics for the Syllogistic Logic with Cardinality Comparison, $S^{\dagger}(card)$ as introduced in \cite{Larry}. 

The alphabet of the language consists of the following: 
\begin{enumerate}
\item raw variables $P_0$, usually denoted by $p, q, \dots$; 
\item complemented variables $\overline{P_0}$, usually denoted by $\bar{p}, \bar{q}, \dots$;
\end{enumerate}
Together they are the nouns of the language: $P = P_0 \cup \overline{P_0}$. 
\begin{enumerate}
\item[(3)] Connectives $\forall(\cdot,\cdot)$, $\exists(\cdot,\cdot)$, $\exists^{\geq}(\cdot,\cdot)$, and $\exists^{>}(\cdot,\cdot)$, where the $\cdot$ is a place holder. 
\end{enumerate}

The sentences of the language are formed from the alphabet by inserting the nouns into the connectives.  For $x, y \in P$ the sentences of $S^{\dagger}(card)$ are:
\begin{enumerate}
\item $\forall(x, y)$, with intended meaning ``all $x$ are $y$''. 
\item $\exists(x, y)$, with intended meaning ``some $x$ are $y$''. 
\item $\exists^{\geq}(x, y)$, with intended meaning ``there are at least as many $x$ as $y$''. 
\item $\exists^{>}(x, y)$, with intended meaning ``there are more $x$ than $y$''. 
\end{enumerate}
\hfill \\ 

\begin{remark}
	We may denote a general sentence as $\phi(x, y)$ so as to talk about the nouns in it. 
\end{remark}

Given a set $P$ of nouns, we define its models. 
\begin{definition}
A \textit{model} $\mathfrak{M} = \langle M, I\rangle$ consists of a finite set $M$ and an interpretation function $I : P \to \mathcal{P}(M)$ such that $I(\bar{x}) = M \setminus I(x)$ for all raw variables $x \in P_0$. 
\end{definition}

\begin{remark}\label{bar_involution}
We extend the bar notation to all nouns by setting $\bar{\bar{p}} \mathrel{:=} p$ for
every raw variable $p \in P_0$. Then, for \emph{every} noun $x \in P$ we have
$I(\bar{x}) = M \setminus I(x)$: this holds by definition when $x \in P_0$, and
when $x = \bar{p}$ for some $p \in P_0$ we have
\[
I(\bar{x}) = I(\bar{\bar{p}}) = I(p) = M \setminus \bigl(M \setminus I(p)\bigr)
= M \setminus I(\bar{p}) = M \setminus I(x).
\]
\end{remark}

Then we can define the satisfaction relation. 

\begin{definition}
Given a sentence $\phi$, we say a model $\mathfrak{M} = \langle M, I\rangle$ \textit{satisfies} $\phi$, denoted $\mathfrak{M}  \models \phi$ if 
\begin{enumerate} 
\item $I(x) \subseteq I(y)$, if $\phi = \forall(x, y)$ for some $x, y \in P$;
\item $I(x) \cap I(y) \neq \varnothing$,  if $\phi = \exists(x, y)$ for some $x, y \in P$;
\item $|I(x)| \geq |I(y)|$,  if $\phi = \exists^{\geq}(x, y)$ for some $x, y \in P$;
\item $|I(x)| > |I(y)|$,  if $\phi = \exists^{>}(x, y)$ for some $x, y \in P$. 
\end{enumerate}
\end{definition}

\begin{remark}
This is well defined since by definition every sentence $\phi$ in the language of $S^{\dagger}(card)$ must fall into one of the four cases. 
\end{remark}

Although there is no negation in the alphabet, every sentence has a \textit{semantic
negation}, as observed in \cite{Larry}:

\begin{definition}\label{sem_neg_def}
	Let $\phi$ be a sentence in the language. Define its \textit{semantic negation} $\bar{\phi}$ to be: 
	\begin{enumerate} 
	\item $\exists(x, \bar{y})$, if $\phi = \forall(x, y)$ for some $x, y \in P$;
	\item $\forall(x, \bar{y})$,  if $\phi = \exists(x, y)$ for some $x, y \in P$;
	\item $\exists^{>}(y, x)$,  if $\phi = \exists^{\geq}(x, y)$ for some $x, y \in P$;
	\item $\exists^{\geq}(y, x)$,  if $\phi = \exists^{>}(x, y)$ for some $x, y \in P$. 
\end{enumerate}
\end{definition}

\begin{lemma}\label{sem_neg}
Let $\mathfrak{M} = \langle M, I\rangle$ be a model and let $\phi$ be a sentence of
$S^\dagger(card)$. Then
\[
\mathfrak{M} \models \bar{\phi} \quad\text{if and only if}\quad \mathfrak{M} \nvDash \phi.
\]
Moreover $\bar{\bar{\phi}} = \phi$.
\end{lemma}

\begin{proof}
By Remark~\ref{bar_involution} we may take $x, y$ to range over all of $P$. We check the
four cases of Definition~\ref{sem_neg_def}.

If $\phi = \forall(x, y)$, then $\mathfrak{M} \nvDash \phi$ iff
$I(x) \not\subseteq I(y)$ iff $I(x) \cap (M \setminus I(y)) \neq \varnothing$ iff
$I(x) \cap I(\bar{y}) \neq \varnothing$ iff
$\mathfrak{M} \models \exists(x, \bar{y}) = \bar{\phi}$.

If $\phi = \exists(x, y)$, then $\mathfrak{M} \nvDash \phi$ iff
$I(x) \cap I(y) = \varnothing$ iff $I(x) \subseteq M \setminus I(y) = I(\bar{y})$ iff
$\mathfrak{M} \models \forall(x, \bar{y}) = \bar{\phi}$.

If $\phi = \exists^{\geq}(x, y)$, then, since $M$ is finite, $\mathfrak{M} \nvDash \phi$
iff $|I(x)| < |I(y)|$ iff $|I(y)| > |I(x)|$ iff
$\mathfrak{M} \models \exists^{>}(y, x) = \bar{\phi}$.

If $\phi = \exists^{>}(x, y)$, then likewise $\mathfrak{M} \nvDash \phi$ iff
$|I(x)| \leq |I(y)|$ iff $\mathfrak{M} \models \exists^{\geq}(y, x) = \bar{\phi}$.

For the last claim, $\overline{\overline{\forall(x,y)}} = \overline{\exists(x,\bar{y})}
= \forall(x, \bar{\bar{y}}) = \forall(x, y)$ by Remark~\ref{bar_involution}, and the
remaining three cases are similar.
\end{proof}

By a \emph{theory} we mean an arbitrary set of sentences; no closure
under semantic consequence is assumed. We can generalise the satisfaction relation to a theory: 

\begin{definition}
Given a theory $\Gamma$, we say a model $\mathfrak{M} = \langle M, I\rangle$ \textit{satisfies} $\Gamma$, denoted $\mathfrak{M}  \models \Gamma$ if for all $\phi \in \Gamma$, we have $\mathfrak{M}  \models \phi$. 
\end{definition}

Then we get the semantical consequence relation. \footnote{See \cite{Larry} for a logical system which is sound and complete with respect to this semantics.}

\begin{definition}
Given a theory $\Gamma$, we say a sentence $\phi$ is a \textit{consequence} of $\Gamma$, denoted $\Gamma \models \phi$ if for all model $\mathfrak{M}$ such that $\mathfrak{M}  \models \Gamma$, we have $\mathfrak{M}  \models \phi$. We write $\psi \models \phi$ for $\{\psi\} \models \phi$. 
\end{definition}

Lastly, for any theory $\Gamma$, we may ask whether there exists a model that satisfies $\Gamma$. 

\begin{definition}
A theory $\Gamma$ is \textit{satisfiable} if there exists a model $\mathfrak{M}$ such that $\mathfrak{M}  \models \Gamma$. Equivalently, $\Gamma$ is satisfiable if there exists a set $M$ such that $\mathfrak{M} = \langle M, I\rangle$ satisfies $\Gamma$ for some interpretation function $I$. In the latter case we will say $\Gamma$ is satisfiable on $M$, the underlying set. 
\end{definition}

\newpage

\section{Main results}
\subsection{Key Observation and Definitions} \hfill \\ 

Consider the theory $\Gamma = \{\exists^{\geq}(x, \bar{x}), \exists^{\geq}(\bar{x}, x)\}$. On which sets is $\Gamma$ satisfiable? \\ 

Suppose there is a model $\mathfrak{M} = \langle M, I\rangle$ with
$\mathfrak{M} \models \Gamma$. Then $|I(x)| \geq |I(\bar{x})|$ and
$|I(\bar{x})| \geq |I(x)|$, so $|I(x)| = |I(\bar{x})|$. Since $I(\bar{x}) = M \setminus I(x)$,
the sets $I(x)$ and $I(\bar{x})$ partition $M$, so
$|M| = |I(x)| + |I(\bar{x})| = 2|I(x)|$. Therefore, $|M|$ is even, and
$|I(x)| = |I(\bar{x})| = \frac{1}{2}|M|$. Thus,
$\{\exists^{\geq}(x, \bar{x}), \exists^{\geq}(\bar{x}, x)\}$ is satisfiable only on sets
of even cardinality.

Conversely, it is easy to see that for any set $M$ with even cardinality, we can define an interpretation function $I$ such that $\langle M, I\rangle \models \Gamma$. Therefore, the theory $\{\exists^{\geq}(x, \bar{x}), \exists^{\geq}(\bar{x}, x)\}$ encodes some information about the structure of our model --- namely, the model has even cardinality.

This naturally raises more questions. If $S^{\dagger}(card)$ can distinguish the even cardinalities, can it similarly distinguish the cardinalities divisible by $3$? More generally, which subsets of $\mathbb N^+$ can occur as the cardinalities of models satisfying a $S^{\dagger}(card)$-theory?

To formulate this question precisely, we adapt the standard notion of the finite spectrum of a first-order sentence to the present language. Since $S^{\dagger}(card)$ does not contain conjunction at the sentence level, it is natural to define spectra for sets of sentences rather than for individual sentences.

\begin{definition}[Spectrum in $S^{\dagger}(card)$]\label{spec_def}
\[
	Spec(\Gamma) = \{n \in \mathbb{N}^+ : \Gamma \text{ has a model of cardinality } n\}.
\]
\end{definition}

\begin{definition}
A subset $U \subseteq \mathbb{N}^+$ is an $S^{\dagger}(card)$-spectrum if $U = Spec(\Gamma)$ for some theory $\Gamma$ of $S^{\dagger}(card)$. 
\end{definition}

Our aim is to characterize all $S^{\dagger}(card)$-spectra. For simplicity, we will simply write ``spectrum'' for $S^{\dagger}(card)$-spectrum in the remainder of the paper, except in the section where we consider the extension of $S^{\dagger}(card)$ by Boolean connectives. Context will make it clear which language we are referring to.

In this language, models of the same cardinality are not distinguishable: 

\begin{lemma}\label{card_essence}
	If $\Gamma$ is satisfiable on a set $M_1$, then for any set $M$ with $|M| = |M_1|$, $\Gamma$ is satisfiable on $M$. 
\end{lemma}

\begin{proof}
	Let $I_1$ be an interpretation function such that $\langle M_1, I_1\rangle \models \Gamma$, Since $|M| = |M_1|$, there exists a bijection $f: M \to M_1$. Then it is easy to verify that $I(x) = f^{-1}[I_1(x)]$ satisfies $\langle M, I\rangle \models \Gamma$. 
\end{proof}

Therefore, the particular choice of an underlying $n$-element set is not important, and we can take the canonical choice $n = \{0, 1, \dots, n-1\}$ as the representing set for all sets with that cardinality.

\begin{corollary}
	For any $\Gamma$, $Spec(\Gamma) = \{n \in \mathbb{N}^+ \mid \mbox{$\Gamma$ is satisfiable on $n = \{0, \dots, n-1\}$} \}$.
\end{corollary}

\begin{proof}
	Easily follows from the previous lemma and definition. 
\end{proof}

\begin{example} 
	\begin{enumerate}
		\item The set of even numbers is the spectrum of $\{\exists^{\geq}(x, \bar{x}), \exists^{\geq}(\bar{x}, x)\}$ by the opening remark. 
		\item The empty set is the spectrum of $\{\exists^{>}(x, x)\}$. In fact, it is the spectrum of any contradictory set of sentences. 
		\item For all $n \in \mathbb{N}^+$, $\{i \mid i \geq n\}$ is the spectrum of $\{\exists^{>}(x_{1}, x_{0}), \dots, \exists^{>}(x_{n}, x_{n-1})\}$. 
	\end{enumerate}
\end{example}

\hfill \\ 

\subsection{Properties of Spectra} \hfill \\ 

We now prove some properties of spectra. First of all, the intersection of spectra is a spectrum. 

\begin{lemma}\label{intersection}
	Let $U_1$ and $U_2$ be spectra. Then $U_1 \cap U_2$ is a spectrum.
\end{lemma}

\begin{proof}
	Let $U_1 = Spec(\Gamma_1)$ and $U_2 = Spec(\Gamma_2)$. Uniformly renaming variables if necessary, we may assume that $\Gamma_1$ and $\Gamma_2$ use disjoint sets of raw variables. This clearly has no effect on their spectra. Then we show that $U_1 \cap U_2 = Spec(\Gamma_1 \cup \Gamma_2)$. 

	$Spec(\Gamma_1 \cup \Gamma_2) \subseteq U_1 \cap U_2$ is immediate, since if $\Gamma_1 \cup \Gamma_2$ is satisfiable on $n$, then both $\Gamma_1$ and $\Gamma_2$ are satisfiable on $n$.
	
	Conversely, let $n\in U_1\cap U_2$, and let $I_1$ and $I_2$ be such that $\langle n, I_1\rangle \models \Gamma_1$ and $\langle n, I_2\rangle \models \Gamma_2$. Since the variables are disjoint, we can combine $I_1$ and $I_2$ into an interpretation satisfying $\Gamma_1\cup\Gamma_2$.
\end{proof}

The next lemma is strictly weaker than the main result we will prove in the next section. Also, the proof for that result does not depend on this one. However, besides the fact that this lemma is found first and can already answer some questions regarding the classification of spectra, we include it here also due to the different construction involved in the proof. 

\begin{lemma}\label{add}
	Let $U \subseteq \mathbb{N}^+$ be a spectrum. If $m, n \in U$, then $m+n \in U$. 
\end{lemma}

\begin{proof}
	Let $U = Spec(\Gamma)$. We prove $\Gamma$ is satisfiable on $m \sqcup n = (m \times \{0\}) \cup (n \times \{1\})$. Let $\langle m, I_1\rangle \models \Gamma$ and $\langle n, I_2\rangle \models \Gamma$. Then for every noun $x$ occurring in $\Gamma$, define $I(x)= I_1(x) \sqcup I_2(x) = (I_1(x) \times\{0\})\cup(I_2(x)\times \{1\})$. 
	
	This is indeed an interpretation function because $I(x) \cup I(\bar{x}) = (I_1(x) \sqcup I_2(x)) \cup (I_1(\bar{x}) \sqcup I_2(\bar{x})) = (I_1(x) \cup I_1(\bar{x})) \sqcup (I_2(x) \cup I_2(\bar{x})) = m \sqcup n$ and $I(x) \cap I(\bar{x}) = (I_1(x) \sqcup I_2(x)) \cap (I_1(\bar{x}) \sqcup I_2(\bar{x})) = (I_1(x) \cap I_1(\bar{x})) \sqcup (I_2(x) \cap I_2(\bar{x})) = \varnothing \sqcup \varnothing = \varnothing$. 
	
	We claim that $\langle m \sqcup n, I\rangle \models \Gamma$. This is proven by cases. Let $\phi \in \Gamma$. 
	\begin{enumerate} 
	\item If $\phi = \forall(x, y)$ for some $x, y \in P$. Since $\langle m, I_1\rangle \models \Gamma$ and $\langle n, I_2\rangle \models \Gamma$, we have $I_1(x) \subseteq I_1(y)$ and $I_2(x) \subseteq I_2(y)$. Thus $I(x) = I_1(x) \sqcup I_2(x) \subseteq I_1(y) \sqcup I_2(y) = I(y)$. 
	\item If $\phi = \exists(x, y)$ for some $x, y \in P$. Since $\langle m, I_1\rangle \models \Gamma$ and $\langle n, I_2\rangle \models \Gamma$, we have $I_1(x) \cap I_1(y) \neq \varnothing$ and $I_2(x) \cap I_2(y) \neq \varnothing$. Thus $I(x) \cap I(y) \neq \varnothing$. 
	\item If $\phi = \exists^{\geq}(x, y)$ for some $x, y \in P$. Then $|I_1(x)| \geq |I_1(y)|$ and $|I_2(x)| \geq |I_2(y)|$. Therefore $|I(x)| = |I_1(x)| + |I_2(x)| \geq |I_1(y)| + |I_2(y)|$. 
	\item If $\phi = \exists^{>}(x, y)$ for some $x, y \in P$. Then $|I_1(x)| > |I_1(y)|$ and $|I_2(x)| > |I_2(y)|$. Therefore $|I(x)| = |I_1(x)| + |I_2(x)| > |I_1(y)| + |I_2(y)|$.
	\end{enumerate}
\end{proof}

From the proof we actually get a more general result: 
\begin{lemma}
For any theory $\Gamma$, if $\mathfrak{M}_1 \models \Gamma$ and $\mathfrak{M}_2 \models \Gamma$, then their disjoint union $(M, I)$ also satisfies $\Gamma$, where $M = (M_1\times\{0\})\cup(M_2\times\{1\})$ and $I(x) = \bigl(I_1(x)\times\{0\}\bigr)
\cup
\bigl(I_2(x)\times\{1\}\bigr)$ for all nouns $x$.
\end{lemma}

\begin{proof}
This is exactly what was shown in the proof of Lemma~\ref{add}.
\end{proof}

With these results, we can already obtain a negative result regarding the spectra, namely, that the odd numbers are not a spectrum. 

\begin{corollary}
	The odd numbers are not a spectrum. 
\end{corollary}

\begin{proof}
	Suppose for contradiction that the positive odd integers form a spectrum $U$. Since $1,3\in U$, Lemma~\ref{add} gives $4=1+3\in U.$ However, $4$ is not odd, a contradiction.
\end{proof}

\begin{remark}
This might be a surprise, since it is tempting to think that if the spectrum of $\Gamma$ is $U$, then the spectrum of $\overline{\Gamma} \mathrel{:=} \{\bar{\phi} \mid \phi \in \Gamma\}$ may give $\mathbb{N}^+\setminus U$. However, this is not the case, and some examples make it clear:

\begin{enumerate}
	\item Let $\Gamma = \{\exists^{\geq}(x, \bar{x}), \exists^{\geq}(\bar{x}, x)\}$. Then $Spec(\Gamma) = \{n \in \mathbb{N}^+ \mid n\text{ is even}\}$, while $Spec(\overline{\Gamma}) = Spec(\{\exists^{>}(\bar{x}, x), \exists^{>}(x, \bar{x})\}) = \varnothing$.
	\item Let $\Gamma = \{\exists^{>}(x, y)\}$. Then $Spec(\Gamma) = \{n \in \mathbb{N}^+ \mid n \geq 1\}$, while $Spec(\overline{\Gamma}) = Spec(\{\exists^{\geq}(y, x)\}) = \mathbb{N}^+$.
	\item Let $\Gamma = \{\exists^{>}(x, y), \exists^{>}(y, z)\}$. Then $Spec(\Gamma) = \{n \in \mathbb{N}^+ \mid n \geq 2\}$, while $Spec(\overline{\Gamma}) = Spec(\{\exists^{\geq}(y, x), \exists^{\geq}(z, y)\}) = \mathbb{N}^+$.
\end{enumerate}

This is because of the order of quantification over interpretations and sentences: by
Lemma~\ref{sem_neg},
\[
n \notin Spec(\Gamma) \iff \forall I\, \exists \phi \in \Gamma\,
\bigl(\langle n, I\rangle \models \bar{\phi}\bigr),
\]
while
\[
n \in Spec(\{\bar{\phi} \mid \phi \in \Gamma\}) \iff \exists I\, \forall \phi \in \Gamma\,
\bigl(\langle n, I\rangle \models \bar{\phi}\bigr).
\]

This is related to Asser's question of whether the complement of every first-order spectrum is again a first-order spectrum~\cite{Asser1955,SpectrumSurvey}. As we just saw, the analogue of Asser's problem in this language has a negative answer: the even positive integers form an $S^\dagger(card)$-spectrum, whereas their complement in $\mathbb N^+$, the odd positive integers, does not.

\end{remark}

\hfill \\ 
\subsection{Characteristic Theories} \hfill \\ 

The above properties answer negatively the question whether all subsets are spectra, thus putting some limit on the expressive power of the language $S^{\dagger}(card)$. However, they are not strong enough to answer the question whether the positive integers divisible by three form a spectrum. We give an attempt to answer this question to motivate the definition of the characteristic theory of a number. 

Let $U = \{3n \mid n \in \mathbb{N}^+\}$, and we want to find a theory $\Gamma$ such that $U = Spec(\Gamma)$. Of course, $3 \in U$. Moreover, since we want to include only numbers divisible by $3$, it is certainly an element with some considerable importance. Therefore, let us try to give as detailed a description of the set $3=\{0,1,2\}$ as the language $S^\dagger(card)$ permits. For this purpose, let us introduce a raw variable for every subset of $3$ --- $P_{-1} = \mathcal{P}(3)$. However, since we can negate the raw variable, for each pair of disjoint sets $S_1$ and $S_2$ with $S_1 \cup S_2 = 3$, we only need to introduce one of them as a raw variable, and the other one will be taken care of by the negation. We choose to introduce the smaller set as a raw variable. Thus, let $P_{0} = \{S \subseteq 3 \mid |S| \leq 1\} = \{\{0\}, \{1\}, \{2\}, \varnothing\}$. Let $P$ be the resulting set of nouns. Let $I$ be the interpretation such that for any $S \in P_0$, $I(S) = S \subseteq 3$ and take the set complement for their negations. In this way, every subset of the universe is named by a noun. Now, let us use the language to say everything that is true of them: let $\Gamma = \{\phi(x, y) \mid x, y \in P, \langle 3, I\rangle \models \phi \}$. This is the best description we can give for the set $3$. We can ask, what is $Spec(\Gamma)$? 

By Lemma~\ref{add}, $U \subseteq Spec(\Gamma)$. On top of that, we know $\langle 3, I\rangle \models \exists^{>}(\{0, 1, 2\}, \{0, 1\})$, $\langle 3, I\rangle \models \exists^{>}(\{0, 1\}, \{0\})$, and $\langle 3, I\rangle \models \exists(\{0\},\{0\})$. Thus, the universe of any model of $\Gamma$ contains sets whose cardinalities satisfy $1\leq |I(\{0\})| < |I(\{0,1\})| < |I(\{0,1,2\})|.$ Therefore, every model of $\Gamma$ has at least three elements, so we have $Spec(\Gamma) \subseteq \{n \mid n \geq 3\}$. 

Can we strengthen the last inclusion to get $Spec(\Gamma) \subseteq \{3n \mid n \geq 1\}$? The answer is no. Consider $4 = \{0, 1, 2, 3\}$ with the same interpretation $I(S) = S$ if $S \in P_0$ and $I(\bar{S}) = 4 \setminus S$ if $\bar{S} \in P_1$. One can check that $ \langle 4, I\rangle \models \Gamma$ (we will give a proof for a more general result below). 

Therefore, even the best description of $3$ in the language cannot distinguish between $3$ and $4$, so it is not possible for any description to do the job (we will also make precise the generalised version of this statement and give a proof).

In this manner, we can prove that the numbers divisible by $3$ do not form a spectrum. A closer analysis shows that in the above arguments, only some properties of $3$ and $4$ are essential. We can thus generalise the arguments for stronger indistinguishability results. It turns out that this is the last step toward a complete classification of spectra in the language $S^{\dagger}(card)$.

We first make precise what we mean by the best description of a given set, and prove that it really is the ``best'' possible. In fact, by Lemma~\ref{card_essence}, only the cardinality of a set matters, so we can choose $n = \{0, 1, \dots, n-1\}$ as the representing set for all sets with that cardinality.

In the following construction, certain subsets of $n$ themselves are taken as raw variables. Thus, the elements of $P_0$ are literally subsets of $n$, and under the canonical interpretation each such raw variable $S$ is interpreted as the subset $S$ itself.

We will denote the collection of all subsets of $n$ with cardinality $k$ by $\binom{n}{k}$.

\begin{definition}\label{char_theory}
Let $n \in \mathbb{N}^+$. If $n$ is odd, then let $P_{0} = \{S \subseteq n \mid |S| < \frac{n}{2}\}$, the negated literals be $P_1$, and the resulting nouns be $P$.  If $n$ is even, then let $\mathcal F_n \subseteq \binom{n}{n/2}$ such that for any $A \in \binom{n}{n/2}$, either $A \in \mathcal F_n$ or $n \setminus A \in \mathcal F_n$, but not both\footnote{For a concrete constructive choice of $\mathcal F_n$, one might take $\left\{
A\in\binom{n}{n/2}:0\in A
\right\}$, since for every $A \in \binom{n}{n/2}$, either $0 \in A$ or $0 \in n \setminus A$.}. Let $P_{0} = \{S \subseteq n \mid |S| < \frac{n}{2}\} \cup \mathcal F_n$, the negated literals be $P_1$, and the resulting nouns be $P$.

Define the interpretation function for $n$ in the most natural way: let $I_n: P \to \mathcal{P}(n)$ be such that $I_n(S) = S$ for all $S \in P_0$, and $I_n(\bar{S}) = n \setminus S$ for all $\bar{S} \in P_1$. Note that $I_n$ is a bijection, so we will denote its inverse by $I_n^{-1}$. Thus, for every $A\subseteq n$, $I_n^{-1}(A)$ is the unique noun whose
interpretation under the canonical interpretation $I_n$ is $A$.

We define the \textit{characteristic theory} of $n$ as $\Gamma_n = \{\phi(x, y) \mid x, y \in P,  \langle n, I_n \rangle \models \phi(x, y)\}$. 
\end{definition}

Any characteristic theory $\Gamma_n$ induces a spectrum $Spec(\Gamma_n)$.

Then $\Gamma_n$ is the best description in $S^{\dagger}(card)$ for $n$ because it excludes other cardinalities as far as the expressive resources of the language permit. Put another way, all potential descriptions of $n$ include the spectrum $Spec(\Gamma_n)$, possibly with some other redundancies: 

\begin{lemma}\label{smallest}
Let $\Gamma$ be any $S^\dagger(card)$-theory. Suppose $n\in Spec(\Gamma)$. Then we have $Spec(\Gamma_n)
\subseteq Spec(\Gamma).$
\end{lemma}

\begin{proof}

Let $m \in Spec(\Gamma_n)$ and we need to show $m \in Spec(\Gamma)$. Let $P$ be the set of nouns used in the construction of $\Gamma_n$ in Definition~\ref{char_theory}, and let $P_\Gamma$ be the set of nouns occurring in $\Gamma$, together with their complements.

Since $n \in Spec(\Gamma)$, let $I_1:P_\Gamma \to \mathcal P(n)$ be an interpretation such that $\langle n,I_1\rangle \models \Gamma$. Recall that $I_n:P\to\mathcal P(n)$ is the canonical interpretation from Definition~3.11, and that $I_n$ is a bijection. Since $m \in Spec(\Gamma_n)$, let $I_0:P\to\mathcal P(m)$ be an interpretation such that $\langle m,I_0\rangle \models \Gamma_n$.

Define $I_2:P_\Gamma\to\mathcal P(m)$ by $I_2=I_0\circ I_n^{-1}\circ I_1$. In other words, for every $x\in P_\Gamma$, let $I_2(x)=I_0(I_n^{-1}(I_1(x)))$. We first show that $I_2$ is an interpretation.

Let $p$ be a raw noun in $P_\Gamma$. Since $I_1$ is an interpretation, $I_1(\bar p)=n\setminus I_1(p)$. Let $a=I_n^{-1}(I_1(p))$. Then $I_n(a)=I_1(p)$. Since $I_n$ is an interpretation, $I_n(\bar a)=n\setminus I_n(a)=n\setminus I_1(p)=I_1(\bar p)$. Since $I_n$ is a bijection, it follows that $I_n^{-1}(I_1(\bar p))=\bar a=\overline{I_n^{-1}(I_1(p))}$. Therefore, since $I_0$ is an interpretation, $I_2(\bar p)=I_0(I_n^{-1}(I_1(\bar p)))=I_0(\overline{I_n^{-1}(I_1(p))})=m\setminus I_0(I_n^{-1}(I_1(p)))=m\setminus I_2(p)$. Hence $I_2$ is an interpretation on $m$.

We now show that $\langle m,I_2\rangle\models\Gamma$. Let $\phi\in\Gamma$ be arbitrary. Since $\langle n,I_1\rangle\models\Gamma$, we have $\langle n,I_1\rangle\models\phi$. There are four cases.

\begin{enumerate}

	\item Suppose $\phi=\forall(x,y)$ for some $x,y\in P_\Gamma$. Then $I_1(x)\subseteq I_1(y)$. Since $I_n(I_n^{-1}(I_1(x)))=I_1(x)$ and $I_n(I_n^{-1}(I_1(y)))=I_1(y)$, we have $\langle n,I_n\rangle\models\forall(I_n^{-1}(I_1(x)),I_n^{-1}(I_1(y)))$. Hence, by the definition of $\Gamma_n$, \\$\forall(I_n^{-1}(I_1(x)),I_n^{-1}(I_1(y)))\in\Gamma_n$. Since $\langle m,I_0\rangle\models\Gamma_n$, we have $I_0(I_n^{-1}(I_1(x)))\subseteq I_0(I_n^{-1}(I_1(y)))$. Therefore $I_2(x)\subseteq I_2(y)$, so $\langle m,I_2\rangle\models\phi$.

	\item Suppose $\phi=\exists(x,y)$ for some $x,y\in P_\Gamma$. Then $I_1(x)\cap I_1(y)\neq\emptyset$. Hence $\langle n,I_n\rangle\models\exists(I_n^{-1}(I_1(x)),I_n^{-1}(I_1(y)))$, so $\exists(I_n^{-1}(I_1(x)),I_n^{-1}(I_1(y)))\in\Gamma_n$. Since $\langle m,I_0\rangle\models\Gamma_n$, we have $I_0(I_n^{-1}(I_1(x)))\cap I_0(I_n^{-1}(I_1(y)))\neq\emptyset$. Therefore $I_2(x)\cap I_2(y)\neq\emptyset$, so $\langle m,I_2\rangle\models\phi$.

	\item Suppose $\phi=\exists^{\geq}(x,y)$ for some $x,y\in P_\Gamma$. Then $|I_1(x)|\geq |I_1(y)|$. Hence $\langle n,I_n\rangle\models\exists^{\geq}(I_n^{-1}(I_1(x)),I_n^{-1}(I_1(y)))$, so $\exists^{\geq}(I_n^{-1}(I_1(x)),I_n^{-1}(I_1(y)))\in\Gamma_n$. Since $\langle m,I_0\rangle\models\Gamma_n$, we have $|I_0(I_n^{-1}(I_1(x)))|\geq |I_0(I_n^{-1}(I_1(y)))|$. Therefore $|I_2(x)|\geq |I_2(y)|$, so $\langle m,I_2\rangle\models\phi$.

	\item Suppose $\phi=\exists^{>}(x,y)$ for some $x,y\in P_\Gamma$. Then $|I_1(x)|>|I_1(y)|$. Hence $\langle n,I_n\rangle\models\exists^{>}(I_n^{-1}(I_1(x)),I_n^{-1}(I_1(y)))$, so $\exists^{>}(I_n^{-1}(I_1(x)),I_n^{-1}(I_1(y)))\in\Gamma_n$. Since $\langle m,I_0\rangle\models\Gamma_n$, we have $|I_0(I_n^{-1}(I_1(x)))|>|I_0(I_n^{-1}(I_1(y)))|$. Therefore $|I_2(x)|>|I_2(y)|$, so $\langle m,I_2\rangle\models\phi$.

\end{enumerate}

Therefore, $\langle m,I_2\rangle\models\Gamma$, and hence $m\in Spec(\Gamma)$ by definition.

\end{proof}

\hfill \\ 
\subsection{Indistinguishability of Cardinalities} \hfill \\

We are now one last step from the final result. What we need are three lemmas concerning the indistinguishability of cardinalities. 

The first two show that for any odd number $n$, its best description cannot distinguish it from any greater number, and that for an even number $n$, its best description cannot distinguish it from any greater even number.

\begin{lemma}\label{o_indiscernible}
	For any odd number $n$,$\{i \mid i \geq n\}\subseteq Spec(\Gamma_n)$.  
\end{lemma}
\begin{proof}
	Let $m \geq n$. We prove that $m \in Spec(\Gamma_n)$. Let $P = P_0 \cup P_1$ be the set of nouns. Define $I: P \to \mathcal{P}(m)$ as follows: $I(S) = I_n(S)$ for all $S \in P_0$ and $I(\bar{S}) = m \setminus S$ for all $\bar{S} \in P_1$. This is clearly an interpretation function for $m$. Note that what this interpretation does is to assign all the additional elements in $m \setminus n$ to the interpretation of each complemented noun.
	
	We now show $\langle m, I\rangle \models \Gamma_n$. Let $\phi \in \Gamma_n$. We have $\langle n, I_n\rangle \models \phi$. 
	
	\begin{enumerate} 
 		 \item $\phi = \forall(x, y)$ for some $x, y \in P$. 
		 
		 If $x, y \in P_0$, then $I(x) = I_n(x) \subseteq I(y) = I_n(y)$. 
		 
		 If $x, y \in P_1$ then, let $x = \bar{a}$ and $y = \bar{b}$, where $a, b \in P_0$. We have $I_n(x) \subseteq I_n(y)$, so $I_n(\bar{a}) \subseteq I_n(\bar{b})$. By construction, this means $n \setminus I_n(a) \subseteq n \setminus I_n(b)$, so  $I_n(b) \subseteq I_n(a)$. Then we have, by our construction,  $I(b) \subseteq I(a)$. Therefore, $I(\bar{a}) = m \setminus I(a) \subseteq I(\bar{b}) = m \setminus I(b)$. I.e., $I(x) \subseteq I(y)$.
		 
		 If $x \in P_0$ and $y \in P_1$, then let $y = \bar{z}$ for some $z \in P_0$. We have $I_n(x) \subseteq I_n(y)$, so $I_n(x) \subseteq n \setminus I_n(z)$ and hence $I_n(x) \cap I_n(z) = \varnothing$. Thus $I(x) \cap I(z) = \varnothing$, so $I(x) \subseteq m \setminus I(z)$, i.e. $I(x) \subseteq I(y)$. 
		 
		 If $x \in P_1$ and $y \in P_0$. Then let $x = \bar{z}$ for some $z \in P_0$. We thus have $n \setminus I_n(z) \subseteq I_n(y)$. Then in $n$, we have $n \setminus z \subseteq y$. However, this is impossible because $|n \setminus z| > |y|$ as we took $P_{0} = \{S \subseteq n \mid |S| < \frac{n}{2}\}$. 
		 Thus, in all possible cases, we have $\langle m, I\rangle \models \phi$. 
		 \item $\phi = \exists(x, y)$ for some $x, y \in P$. This is clear because the new interpretation function only adds elements to the interpretation of some variables, so non-empty intersections remain non-empty. 
 		 \item $\phi = \exists^{\geq}(x, y)$ for some $x, y \in P$. 
		 
		  If $x, y \in P_0$. This is clear as there is no change to the interpretation. 
		  
		 If $x, y \in P_1$. The new interpretation function adds the same elements to the interpretation of both variables so there is no change regarding the order of their cardinality. 
		 
		 If $x \in P_0$ and $y \in P_1$ This is impossible as we put $P_{0} = \{S \subseteq n \mid |S| < \frac{n}{2}\}$.
		 
		  If $x \in P_1$ and $y \in P_0$. This still holds since the new elements are added only to the interpretation of $x$, not to that of $y$. 
  		 \item $\phi = \exists^{>}(x, y)$ for some $x, y \in P$.
		 Almost the same as the previous case. 
 	\end{enumerate}
	
	Therefore, $\langle m, I\rangle \models \Gamma_n$ and hence $m \in Spec(\Gamma_n)$. 
\end{proof}

\begin{lemma}\label{e_indiscernible}
	For any even number $n$, we have $\{2i \mid 2i \geq n \} \subseteq Spec(\Gamma_n)$.
\end{lemma}

\begin{proof}

We prove by induction that every even $m\geq n$ belongs to $Spec(\Gamma_n)$. The base case $m=n$ is immediate, since $\langle n,I_n\rangle\models\Gamma_n$ by definition.

For the inductive step, suppose $m\in Spec(\Gamma_n)$, where $m\geq n$ is even, and let $I:P\to\mathcal P(m)$ be an interpretation such that $\langle m,I\rangle\models\Gamma_n$. We construct an interpretation $I':P\to\mathcal P(m+2)$ such that $\langle m+2,I'\rangle\models\Gamma_n$.

Recall that $\mathcal F_n$ is the family of subsets of $n$ of cardinality $n/2$ chosen in Definition~\ref{char_theory}, such that for every $A\subseteq n$ with $|A|=n/2$, exactly one of $A$ and $n\setminus A$ belongs to $\mathcal F_n$.

We need to add two new elements, $m$ and $m+1$. For each raw noun $S\in P_0$, let
\[
I'(S)\mathrel{:=}
\begin{cases}
I(S), & \text{if } |S|<\frac{n}{2},\\[2mm]
I(S)\cup\{m\}, & \text{if } S\in\mathcal F_n.
\end{cases}
\]
For every complemented noun $\bar S\in P_1$, let $I'(\bar S)\mathrel{:=}(m+2)\setminus I'(S)$. Then $I'$ is an interpretation on $m+2$ by construction.

Notice what happens to the cardinality of each noun at this step. If $|I_n(x)|<n/2$, then no new element is added to $I(x)$. If $|I_n(x)|=n/2$, then exactly one new element is added to $I(x)$: if $x=S$ for some $S\in\mathcal F_n$, it receives $m$, while its complement $\bar S$ receives $m+1$. Finally, if $|I_n(x)|>n/2$, then $x$ is the complement of a raw noun of cardinality less than $n/2$, and hence both $m$ and $m+1$ are added to $I(x)$.

We now show that $\langle m+2,I'\rangle\models\Gamma_n$. Let $\phi\in\Gamma_n$. We have $\langle n,I_n\rangle\models\phi$ by definition and $\langle m,I\rangle\models\phi$ by inductive hypothesis. There are four cases.

\begin{enumerate}
	\item Suppose $\phi=\forall(x,y)$ for some $x,y\in P$. By the induction hypothesis, $I(x)\subseteq I(y)$, so we only need to check that the newly added elements do not destroy the inclusion.

	If $|I_n(x)|<n/2$, then no new element is added to $I(x)$, so the inclusion is preserved.

	If $|I_n(x)|>n/2$, then $I_n(x)\subseteq I_n(y)$ implies $|I_n(y)|>n/2$. Hence both new elements are added to both $I(x)$ and $I(y)$, so the inclusion is preserved.

	Finally, suppose $|I_n(x)|=n/2$. Since $I_n(x)\subseteq I_n(y)$, either $|I_n(y)|>n/2$, in which case both new elements are added to $I(y)$, or $|I_n(y)|=n/2$. In the latter case, $I_n(x)\subseteq I_n(y)$ and $|I_n(x)|=|I_n(y)|$ imply $I_n(x)=I_n(y)$. Since $I_n$ is a bijection, $x=y$, so the same new element is added to both sides. Thus in every case $I'(x)\subseteq I'(y)$.

	\item Suppose $\phi=\exists(x,y)$ for some $x,y\in P$. By the induction hypothesis, $I(x)\cap I(y)\neq\emptyset$. Since the construction only adds elements to the interpretations of nouns and never removes any, we still have $I'(x)\cap I'(y)\neq\emptyset$.

	\item Suppose $\phi=\exists^{\geq}(x,y)$ for some $x,y\in P$. Since $\phi\in\Gamma_n$, we have $|I_n(x)|\geq|I_n(y)|$. At each inductive step, a noun receives $0$, $1$, or $2$ new elements according as its cardinality under $I_n$ is less than, equal to, or greater than $n/2$. Therefore, $x$ receives at least as many new elements as $y$. By the induction hypothesis, $|I(x)|\geq|I(y)|$, and hence $|I'(x)|\geq|I'(y)|$.

	\item Suppose $\phi=\exists^{>}(x,y)$ for some $x,y\in P$. Again, $|I_n(x)|>|I_n(y)|$, so $x$ receives at least as many new elements as $y$. Since by the induction hypothesis $|I(x)|>|I(y)|$, adding at least as many elements to $I(x)$ as to $I(y)$ preserves the strict inequality. Hence $|I'(x)|>|I'(y)|$.
\end{enumerate}

Therefore $\langle m+2,I'\rangle\models\Gamma_n$, so $m+2\in Spec(\Gamma_n)$. By induction, every even $m\geq n$ belongs to $Spec(\Gamma_n)$.

\end{proof}

The final lemma says, if a description cannot distinguish an even number from a larger odd number, then neither could this description distinguish them from anything in between. In other words, any theory $\Gamma$ describing an even number $n$ can distinguish it from $n+1$ only if $\Gamma$ says ``this number is even!''.

\begin{lemma}\label{eo_indiscernible}
	Let $U$ be a spectrum. If an even number $m \in U$, an odd number $n \in U$, and $n > m$, then we have $i \in U$ for all $m \leq i \leq n$. 
\end{lemma}
\begin{proof}
	For an even $i\geq m$, by Lemma~\ref{smallest} we have that $Spec(\Gamma_m)\subseteq U$ and by Lemma~\ref{e_indiscernible} we have $i\in Spec(\Gamma_m).$ Hence $i\in U$.
	
	Thus let $i$ be odd. By Lemma~\ref{smallest} and~\ref{o_indiscernible}, it suffices to prove $m +1 \in U$. 
	
	Let $U = Spec(\Gamma)$. Let $I_0$ be such that $\langle m, I_0\rangle \models \Gamma$ and $I_1$ be such that $\langle n, I_1\rangle \models \Gamma$. Let $P$ be the set of nouns occurring in $\Gamma$, together with their complements. Let $P = Q_0 \cup Q_1$ where $Q_0 = \{x \in P \mid |I_1(x)| < \frac{1}{2}n\}$ and $Q_1 = P \setminus Q_0 = \{x\in P:|I_1(x)|> \frac{1}{2}n\}$. 
	
	Since $n$ is odd, no noun has interpretation of cardinality exactly $\frac{1}{2}n$. Moreover, $|I_1(x)|+|I_1(\bar x)|=n.$ Hence exactly one of $x$ and $\bar x$ belongs to $Q_0$, and the other belongs to $Q_1$. Therefore, renaming variables if necessary, let $Q_0 = \{x_0, \dots, x_l, \dots\}$ and $Q_1 = \{\bar{x_0}, \dots, \bar{x_l}, \dots\}$. 
	
	We are going to add one additional element $m$ to the set $m = \{0, \dots, m-1\}$ and modify the interpretation $I_0$ so that for each raw variable $x$, we either give $m$ to the interpretation of $x$ or to that of $\bar{x}$. Things may go wrong if adding this number changes the cardinality order, or make disjoint sets intersect, so some caution is necessary. 
	
	Define $I: P \to \mathcal{P}(m+1)$ as follows: 
	\[
	I(z) = 
	\begin{cases}
    	 I_0(z),& \mbox{if } z \in Q_0 \\
    	I_0(z) \cup \{m\},              & \mbox{if } z \in Q_1 \\
	\end{cases}
	\]
	It is routine to check that this is an interpretation function, and here we need the fact that $Q_0$ and $Q_1$ are of these forms: $Q_0 = \{x_0, \dots, x_l, \dots\}$ and $Q_1 = \{\bar{x_0}, \dots, \bar{x_l}, \dots\}$. 
	
	We show $\langle m+1, I\rangle \models \Gamma$. Let $\phi \in \Gamma$.	
	\begin{enumerate} 
 		 \item $\phi = \forall(x, y)$ for some $x, y \in P$. 
		 If $x, y \in Q_0$ Then nothing is changed. 	
		 	 
		 If $x, y \in Q_1$ this case follows since the new interpretation function adds the same elements to the interpretation of both variables.		 
		 
		If $x \in Q_0$ and $y \in Q_1$. This is clear since the new function adds the only new elements to the interpretation of all variables in $Q_1$.

		If $x \in Q_1$ and $y \in Q_0$. This is impossible since $\langle n, I_1\rangle \models \forall(x, y)$ we have $I_1(x) \subseteq I_1(y)$ and hence $|I_1(x)| \leq |I_1(y)|$, contradicting $|I_1(x)| > \frac{n}{2} > |I_1(y)|$.
		 
		 Thus, in all possible cases, we have $\langle m+1, I\rangle \models \phi$. 
		 \item $\phi = \exists(x, y)$ for some $x, y \in P$. This is clear because the new interpretation function only adds elements to the interpretation of some variables, so non-empty intersections remain non-empty. 
 		 \item $\phi = \exists^{\geq}(x, y)$ for some $x, y \in P$. 
		 
		  If $x, y \in Q_0$. Then nothing is changed. 
		  
		 If $x, y \in Q_1$. One new element is added to both interpretations, so the cardinality comparison is preserved. 
		 
		 If $x \in Q_0$ and $y \in Q_1$ This is impossible because $|I_1(x)|<n/2<|I_1(y)|$, but $\langle n,I_1\rangle\models\exists^{\geq}(x,y)$ would require $|I_1(x)|\geq|I_1(y)|$.

		 If $x \in Q_1$ and $y \in Q_0$ This still holds because the new interpretation function adds the new element only to the interpretation of $x$. 
  		 \item $\phi = \exists^{>}(x, y)$ for some $x, y \in P$. The same case analysis applies. The case $x\in Q_0$ and $y\in Q_1$ is impossible in the $n$-element model; adding one new element to both sides preserves a strict inequality; and adding it only to the interpretation of $x$ also preserves strict inequality. Hence $|I(x)|>|I(y)|$. 
 	\end{enumerate}

	Therefore $\langle m+1, I\rangle \models \Gamma$, so $m + 1 \in Spec(\Gamma) = U$, which is what remained to be shown.
\end{proof}

Combining the above gives our main result for this paper. Though let us first recall the notation we introduced in the first section. 

For $k\in\mathbb{N}^+$, let $T_k=\{n\in\mathbb{N}^+:n\geq k\}$ and $E_k=\{n\in\mathbb{N}^+:n\geq k
\text{ and }n\text{ is even}\}.$ Note that for odd $k$, we have $E_k = E_{k+1}$. Then we have:

\begin{theorem}[Spectrum Classification]\label{main}
The $S^\dagger(card)$-spectra are precisely $\varnothing, T_k, E_k$, where $k\in\mathbb{N}^+$.
\end{theorem}

\begin{proof}

	We first check that each of the listed sets is a spectrum.
	\begin{enumerate} 
		\item If $U=\varnothing$, then $U=Spec(\{\exists^{>}(x,x)\})$, so the result holds.
		
		\item Suppose $U=T_k$. Then $U$ is the spectrum of $\{
	\exists^{>}(x_1,x_0),
	\ldots,
	\exists^{>}(x_k,x_{k-1})
	\}.$

	\item Suppose $U=E_k$. Then $U$ is the spectrum of $\{
	\exists^{>}(x_1,x_0),
	\ldots,
	\exists^{>}(x_k,x_{k-1}),\\
	\exists^{\geq}(x,\bar x),
	\exists^{\geq}(\bar x,x)
	\}$, where $x$ is a fresh raw noun distinct from $x_0,\dots,x_k$. 
	\end{enumerate}

	It remains to show that every spectrum is $\varnothing$, $T_k$ or $E_k$ for some
	$k \in \mathbb{N}^+$. Let $U \subseteq \mathbb{N}^+$ be any non-empty spectrum. By Well-Ordering Principle, let $m$ be the least element in $U$. 
	
	If $m$ is odd, then by Lemma~\ref{smallest}, we have $Spec(\Gamma_m) \subseteq U$. Now by Lemma~\ref{o_indiscernible}, $\{i \mid i \geq m\}\subseteq Spec(\Gamma_m)$, so $\{i \mid i \geq m\}\subseteq U$. However, since $m$ is the least element of $U$, it follows that $U = \{i \mid i \geq m\}$. 
	
	If $m$ is even and $U$ contains even numbers only. By Lemma~\ref{smallest}, we have $Spec(\Gamma_m) \subseteq U$ and by Lemma~\ref{e_indiscernible} $\{2i \mid 2i \geq m \} \subseteq Spec(\Gamma_m)$. Thus $\{2i \mid 2i \geq m \} \subseteq U$, and we have $U = \{2i \mid 2i \geq m \}$ since $m$ is the least element and $U$ only contains even numbers. 
	
	If $m$ is even and $U$ also contains an odd number $n >m$. Then as in the first case, we have by  Lemma~\ref{smallest} and~\ref{o_indiscernible} that $\{i \mid i \geq n\}\subseteq U$. Moreover, by Lemma~\ref{eo_indiscernible} $\{i \mid m \leq i \leq n\} \subseteq U$. Therefore $\{i \mid i \geq m\} \subseteq U$, and, again, since $m$ is the least element, $U = \{i \mid i \geq m\}$. 
	
	Hence, all non-empty spectra are either of the shape $\{i \mid i \geq m\}$ or of the shape $\{2i \mid 2i \geq m \}$ for some $m \in \mathbb{N}^+$.
\end{proof}

We now have the full picture of the spectra in this language. Moreover, the three
families of theories exhibited in the first half of the proof of Theorem~\ref{main} are
all finite, so every spectrum is the spectrum of a finite theory.

\begin{corollary}\label{finite}
	Let $U \subseteq \mathbb{N}^+$ be any spectrum. Then there is a finite theory
	$\Gamma_0$ such that $Spec(\Gamma_0) = U$.
\end{corollary}

\begin{proof}
	This was shown in the proof of Theorem~\ref{main}.
\end{proof}


\newpage
\section{Adding Boolean Connectives}
We now move on to consider the situation where we add the Boolean connectives to the language. To be specific, we treat sentences in the language $S^{\dagger}(card)$ as atomic ones, and build the new language by recursively adding the connectives $\land$, $\lor$, and $\neg$. The semantics for the new language is the obvious one. 

Every sentence of the extended language is a finite Boolean combination of
$S^{\dagger}(card)$-sentences, so by propositional logic it is equivalent to one in conjunctive
normal form, in which $\neg$ occurs only in front of $S^{\dagger}(card)$-sentences. By
Lemma~\ref{sem_neg}, each such occurrence $\neg \phi$ may be replaced by the semantic
negation $\bar{\phi}$, which is again an $S^{\dagger}(card)$-sentence. Finally, by the lemma below,
a conjunct $\phi \wedge \psi$ of a theory may be replaced by the two sentences
$\phi, \psi$. Thus, in the presence of disjunction, neither negation nor conjunction
affects which sets arise as spectra, and we may assume that every sentence of a theory is
a finite disjunction of $S^{\dagger}(card)$-sentences.

\begin{lemma}
For every theory $\Gamma$, $Spec(\Gamma\cup\{\phi\wedge\psi\})
= Spec(\Gamma\cup\{\phi,\psi\}).$ 
\end{lemma}

\begin{proof}
	Just use the same interpretation function.
\end{proof}

Therefore, to characterise the spectra $Spec(\Gamma)$ of $\mathbb{N}^+$ in this new language, we may assume that all sentences $\phi \in \Gamma$ are in the form of $\phi = \bigvee_{i < n} \phi_i$ where each $\phi_i$ is a sentence in $S^{\dagger}(card)$. Then if $\mathfrak{M} \models \Gamma$, we must have that for all $\phi \in \Gamma$, $\mathfrak{M} \models \phi$, and hence there must be some $i < n$ such that $\mathfrak{M} \models \phi_i$, and vice versa.

\begin{definition}
Let $\Gamma=\{\phi_i\mid i\in I\}$ be a theory, where $I$ is some index set. Suppose that for every $i\in I$,
$\phi_i=\bigvee_{j<n_i}\phi_{ij}$ for some $n_i\in\mathbb N^+$, where each $\phi_{ij}$ is a sentence in $S^\dagger(card)$. We say $\Delta$ is an \textit{atomic component} of $\Gamma$ if $\Delta=\{\psi_i\mid i\in I\}$, where for every $i\in I$, $\psi_i=\phi_{ij}$ for some $j<n_i$. Let $At(\Gamma)$ be the set of all atomic components of $\Gamma$.
\end{definition}

Using the same notation as in the above definition, we have:

\begin{lemma}
For any model $\mathfrak M$, we have $\mathfrak M\models\Gamma$ iff there is some $\Delta\in At(\Gamma)$ such that $\mathfrak M\models\Delta$.
\end{lemma}

\begin{proof}
$(\Rightarrow)$ Suppose $\mathfrak M\models\Gamma$. Then for every $i\in I$, we have $\mathfrak M\models\phi_i$. Since $\phi_i=\bigvee_{j<n_i}\phi_{ij}$, the set $\{j<n_i\mid\mathfrak M\models\phi_{ij}\}$ is nonempty. Let $k$ be its least element, and let $\psi_i=\phi_{ik}$. Then $\Delta=\{\psi_i\mid i\in I\}$ is an atomic component of $\Gamma$ and $\mathfrak M\models\Delta$.

$(\Leftarrow)$ Suppose $\mathfrak M\models\Delta$ for some $\Delta\in At(\Gamma)$. By definition, $\Delta=\{\psi_i\mid i\in I\}$, where for every $i\in I$, $\psi_i=\phi_{ij}$ for some $j<n_i$. Since $\mathfrak M\models\Delta$, we have $\mathfrak M\models\psi_i$ for every $i\in I$. Hence, for every $i\in I$, at least one disjunct of $\phi_i$ is true in $\mathfrak M$, so $\mathfrak M\models\phi_i$. Therefore $\mathfrak M\models\Gamma$.
\end{proof}

\begin{definition}\label{basis}
Let
\[
\mathcal B
=
\{E_2\}\cup\{T_n:n\in\mathbb{N}^+\}.
\]
We call $\mathcal B$ the \emph{spectral subbasis}.
\end{definition}

Let $\mathcal{T}$ be the set of arbitrary unions of finite intersections of the spectral subbasis. That is, the spectral subbasis is a sub-basis of the topology $\mathcal{T}$ on $\mathbb{N}^+$ --- or, the spectra of $S^{\dagger}(card)$ form a basis for this topology (see, e.g., \cite{topology} for the definition of basis and sub-basis for a topology).

Let $\mathcal{U}$ be the set of all spectra of $\mathbb{N}^+$ in the new language. We will show that $\mathcal{U} = \mathcal{T}$.

\begin{theorem}\label{topology_thm}
	$\mathcal{U} = \mathcal{T}$. 
	
	Explicitly,
	\[
	\mathcal{U} = \mathcal{T} = \{\varnothing\}
	\cup \{T_a : a \in \mathbb{N}^+\}
	\cup \{E_b : b \in \mathbb{N}^+\}
	\cup \{T_a \cup E_b : a, b \in \mathbb{N}^+,\ b < a\},
	\]
\end{theorem}

\begin{proof}
	\begin{enumerate}
	\item[$\mathcal{U} \subseteq \mathcal{T}$]: 
	Let $Spec(\Gamma) \in \mathcal{U}$. Then it follows from the previous lemma that $Spec(\Gamma) = \bigcup_{\Delta \in At(\Gamma)} Spec(\Delta)$, and the result follows as $Spec(\Delta)$ is a spectrum in the original language $S^{\dagger}(card)$. 

	\item[$\mathcal{T} \subseteq \mathcal{U}$]:
	We first deal with finite unions. By Lemma~\ref{intersection}, every finite intersection of elements of $\mathcal B$ is an $S^\dagger(card)$-spectrum. Hence every element of $\mathcal T$ is an arbitrary union of $S^\dagger(card)$-spectra. Let $U \in \mathcal{T}$ be such that $U = \bigcup_{i \in I} Spec(\Gamma_{i})$ for some finite index set $I$, and we may assume $\Gamma_{i}$ is an $S^{\dagger}(card)$-theory. By Corollary~\ref{finite} and taking conjunctions, we may further assume that each $\Gamma_{i}$ contains a single sentence $\phi_i$ in the new language. Then clearly $U$ is the spectrum of $\{\bigvee_{i \in I} \phi_i\}$. 

	For arbitrary unions, recall from Theorem~\ref{main} that every nonempty $S^\dagger(card)$-spectrum is either a tail $T_k$ or an even tail $E_k$. Let $U=\bigcup_{i\in I}U_i$ be an arbitrary union of $S^\dagger(card)$-spectra.

	If no nonempty spectrum occurs in the union, then $U=\varnothing$.

	If some tails $T_k$ occur, let $a$ be the least such that $T_a$ occurs in the union. Then $\bigcup\{T_k:T_k\text{ occurs among the }U_i\}=T_a.$ 

	If some even tails occur, let $b$ be the least threshold among them. Then $\bigcup\{E_k:E_k\text{ occurs among the }U_i\}=E_b.$ 

	Hence, every arbitrary union has one of the forms $\varnothing, T_a, E_b, T_a\cup E_b.$ In particular, every arbitrary union of original spectra is a finite union. By the finite-union argument above, each such set is a spectrum in the Boolean extension.

	\end{enumerate}
\end{proof}

Therefore, we conclude that the collection of spectra in the new language is exactly the topology generated by that in the original language. 

\begin{remark}
The topology $\mathcal{T}$ is Alexandrov: it is closed under arbitrary intersections as
well as arbitrary unions, and the least open set containing $n$ is $T_n$ if $n$ is odd
and $E_n$ if $n$ is even. Also,
$\mathcal{T}$ is $T_0$ but not $T_1$.
\end{remark}

\newpage
\section{Reflections and Further Implications}

\subsection{Why Parity but Not Divisibility by Three?} \hfill \\ 
By Theorem~\ref{main}, the only restrictions that $S^\dagger(card)$ can impose on the cardinality of a finite universe are (1) whether they are even --- achieved by claiming the cardinality of one variable is equal to its complement; and (2) whether they contain at least $n$ elements --- achieved by claiming there are at least $n+1$ distinct nouns interpreted as sets with different sizes. 

Why can it not say more? As the proofs show, the reason is that although the language can express ``one set is larger than another in cardinality'', it cannot provide the crucial detail of ``larger, but by how much?''. Knowing this, it can be seen that by adding the disjunctive connective $\lor$ to the language at the variable level, we can fix this issue. 

To be precise, let the \emph{$\vee$-nouns} be generated from $P$ by closing under the
binary operation $\vee$ and under complementation. Extend the interpretation $I$ to
all $\vee$-nouns by $I(x \vee y) = I(x) \cup I(y)$ and $I(\bar{x}) = M \setminus I(x)$. 

Consider the following theory: \[\Lambda_3= \{
\forall(x,\bar y),
\forall(x,\bar z),
\forall(y,\bar z),
\exists^{\geq}(x,y),
\exists^{\geq}(y,x),
\exists^{\geq}(y,z),
\exists^{\geq}(z,y),
\exists^{\geq}(y\vee z,\bar x)
\}\]

The first three sentences say that $x,y,z$ are pairwise disjoint.
Hence $I(y)\cup I(z)\subseteq I(\bar x).$
The last sentence says $|I(y)\cup I(z)|\geq |I(\bar x)|$. Since the sets are finite, it follows that $I(y)\cup I(z)=I(\bar x).$ Thus $M=I(x)\sqcup I(y)\sqcup I(z).$

The remaining cardinality comparisons imply $|I(x)|=|I(y)|=|I(z)|.$ Therefore, $|M|
=
|I(x)|+|I(y)|+|I(z)|
=
3|I(x)|.$
Hence every model of $\Lambda_3$ has cardinality divisible by $3$.

Conversely, every finite set of cardinality $3k$ can be partitioned into three sets of cardinality $k$, giving a model of $\Lambda_3$. Therefore, $Spec(\Lambda_3)= \{3k:k\in\mathbb{N}^+\}$.

Hence, we see that the sentence $\exists^{\geq}(y\vee z,\bar x)$ which says that ``the union of $y$ and $z$ is at least as large as $\bar{x}$'' is not expressible in $S^{\dagger}(card)$. This sentence is expressively strong in the sense that it allows us, when combining with the other sentences, to say not only that ``$\bar{x}$ is larger than $y$'', but it tells us that ``$\bar{x}$ is larger than $y$ by the size of $z$''. 

Another possibility is to add the conjunctive connective $\wedge$ at the noun level: let
the \emph{$\wedge$-nouns} be generated from $P$ by closing under $\wedge$ and
complementation, with $I(x \wedge y) = I(x) \cap I(y)$. This also allows us to express
divisibility by three\footnote{In fact, the $\lor$-extended and $\land$-extended noun languages are interdefinable: $I(x\vee y)=I(\overline{\bar{x}\wedge\bar{y}})$.}.

Consider the theory
\[\Lambda_3' =
\{\forall(x,\bar y),\forall(x,\bar z),\forall(y,\bar z),
\exists^{\geq}(x,y),\exists^{\geq}(y,x),
\exists^{\geq}(y,z),\exists^{\geq}(z,y),
\exists^{\geq}(y,\bar x\land\bar z)\}.\]

The first three sentences say that $x,y,z$ are pairwise disjoint. In particular, $I(y)\subseteq I(\bar x)\cap I(\bar z)=I(\bar x\land\bar z)$. Now since the last sentence says $|I(y)|\geq|I(\bar x\land\bar z)|$ and the sets are finite, it follows that $I(y)=I(\bar x)\cap I(\bar z)$. Therefore, every element outside both $I(x)$ and $I(z)$ belongs to $I(y)$, and hence $M=I(x)\sqcup I(y)\sqcup I(z)$.

The remaining cardinality comparisons imply $|I(x)|=|I(y)|=|I(z)|$. Therefore, $|M|=|I(x)|+|I(y)|+|I(z)|=3|I(y)|$, so every model of $\Lambda_3'$ has cardinality divisible by $3$.

Conversely, if $|M|=3k$, partition $M$ into three sets $X,Y,Z$ of cardinality $k$, and interpret $x,y,z$ by $X,Y,Z$, respectively. Then all the sentences of $\Lambda_3'$ are satisfied. Therefore $Spec(\Lambda_3')=\{3k:k\in\mathbb N^+\}$.

These examples illustrate how the addition of noun-level union or intersection allows the language, in effect, to say not merely that one set is larger than another, but ``larger by how much''. This additional expressive resource is sufficient to express three-foldness, which is impossible in the base language.

\hfill \\ 
\subsection{Expressive Power of Natural Logics} \hfill \\ 
We can take the collection of spectra as one measure of the expressive power of natural-logical systems whose semantics is restricted to finite models. 

As we have seen, sentence-level conjunction does not enlarge the collection of theory spectra, since a theory already requires all of its member sentences to hold simultaneously. The extension obtained in Section~4 is given by sentence-level disjunction, which allows finite unions of spectra, and the special form of the original spectra then yields the topology described there.

One possible direction for further work is to consider analogous spectrum questions for languages admitting infinite models. In that setting one could study $Spec_\infty(\Gamma) = \{\kappa:
\Gamma\text{ has a model of cardinality }\kappa\}$, 
where $\kappa$ ranges over infinite cardinals. It would be interesting to investigate such spectra for syllogistic logics with cardinality comparisons on infinite sets, such as the system studied by Moss and Topal~\cite{CardComp}.

It is also instructive to compare $S^\dagger(card)$ with the fragment obtained
by dropping the two cardinality comparisons, so that only $\forall(x,y)$ and
$\exists(x,y)$ remain. There, every spectrum is $\varnothing$ or a tail $T_k$: suppose $\langle M, I\rangle \models \Gamma$ and $a \in M$. We can add a new element $a'$ to $M$ and let $I'(x) = I(x) \cup \{a'\}$ if $a \in I(x)$, and $I'(x) = I(x)$ otherwise. This is an interpretation because $a$ belongs to exactly one of $I(x)$ and $I(\bar x)$, so $a'$ is added to exactly one of $I'(x)$ and $I'(\bar x)$. Then it is easy to see that $\langle M \cup \{a'\}, I'\rangle \models \Gamma$, so $Spec(\Gamma)$ is upward closed, and thus it is either $\varnothing$ or $T_k$, where $k$ is its least element. Therefore, it is the addition of cardinality comparisons that allows the language to express parity.

\hfill \\ 
\subsection{Natural Logic and Human Reasoning} \hfill \\ 
As Moss~\cite{Larry} mentions, one purpose of formalising natural language in the way Natural Logic does is to ``describe logical tools that seem to be implicit in the human reasoning facility.''  

From this perspective, the present classification identifies a precise expressive limitation of $S^\dagger(card)$: cardinality comparison together with noun complementation suffices to detect parity,
but not higher divisibility conditions such as divisibility by $3$.

The result may therefore help isolate which additional formal resources are required to represent richer numerical reasoning within a natural-logic framework. 

\hfill \\ 
\subsection{Model Theory} \hfill \\ 
The construction of the characteristic theory $\Gamma_n$ is analogous to the method of diagrams in first-order model theory. Recall that, for
a model $\mathfrak{M}$ in a first-order language, one may expand the language by adding a constant symbol for each element of $M$ and write down the formulas true of these named elements.

\begin{definition}
	Let $\mathfrak{M}$ be a model in language $L$. Let $L_{M}$ be the extension of $L$ by adding a constant symbol for each element of $M$ and $\mathfrak{M}_M$ be a model in this new language with the natural interpretation. Then the complete diagram of $\mathfrak{M}$ is the set $CDiag(\mathfrak{M}) = \{\phi \mid \mbox{$\phi$ is an $L_{M}$ sentence and } \mathfrak{M}_{M} \models \phi\}$. 
\end{definition}
	
Our construction follows a similar idea at a different level. Rather
than naming each element of the finite universe, we introduce nouns
naming each subset of the universe and record every inclusion,
nonempty-intersection, and cardinality-comparison statement about
those subsets that is expressible in $S^\dagger(card)$.

In first-order model theory, we have the following result:

\begin{theorem}
Let $\mathfrak M$ and $\mathfrak N$ be models in $L$. Then $\mathfrak M$ can be elementarily embedded into $\mathfrak N$ iff there is an expansion $\mathfrak N_M$ of $\mathfrak N$ to $L_M$ such that $\mathfrak N_M\models CDiag(\mathfrak M)$.
\end{theorem}

The analogous result in our setting is Lemma~\ref{smallest}: if a theory $\Gamma$ has an
$n$-element model, then $Spec(\Gamma_n) \subseteq Spec(\Gamma)$. Equivalently,
$Spec(\Gamma_n)$ is the $\subseteq$-least spectrum containing $n$, so $\Gamma_n$ records
exactly as much about the cardinality $n$ as the expressive resources of $S^\dagger(card)$ permit. 

The difference between our characteristic theories and the diagrams is also worth noting: our nouns name subsets of the universe rather than
its elements, so $\Gamma_n$ is a diagram at one type level up.

(See, for example, Kirby~\cite{MT} and Tent and Ziegler~\cite{MT2} for more applications of the method of diagram in FOL model theory).

\hfill \\ 

{\bf Acknowledgements.} I would like to thank Lawrence S. Moss for his lectures on natural logic and for supervising the project from which this paper grew. His suggestions and insightful discussions were invaluable in the development of this work. I am also grateful to Nicholas Ramsey for his valuable feedback on an earlier version of the paper, which helped me substantially clarify its structure and situate the results within the literature on spectra. I would also like to thank the attendees at my project presentation at the Institute for Logic, Language and Computation (ILLC), where this project was developed, for their helpful feedback and discussion. I am also grateful to an anonymous referee whose detailed comments on an earlier version of this paper led to substantial improvements in its motivation and exposition. Lastly, I would like to thank Cheng Liao for giving me some valuable advice on revising the paper. 

\medskip
\noindent{\it AI Disclosure.}
The author used OpenAI's ChatGPT and Anthropic's Claude during later stages of the finalisation and revision of this manuscript for editorial and critical assistance, including suggestions concerning exposition, organization, wording, and the presentation of mathematical arguments. The central mathematical results and proofs of the paper were developed, proved, and written up by the author before summer 2024, without any use of AI. All AI-assisted suggestions were independently evaluated and, where incorporated, revised or verified by the author. The author is entirely responsible for the mathematical and scientific content of the paper and for its compliance with the journal's authorship policy.


\bibliographystyle{amsplain}

\end{document}